\documentclass[a4paper, 10pt]{article} \usepackage[utf8]{inputenc} \usepackage{mathtools} \usepackage[T2A]{fontenc} \usepackage{hyperref} \usepackage{dsfont} \usepackage[T2A]{fontenc} \usepackage{amsmath,amssymb,amsthm} \usepackage[a4paper, lmargin = 3cm, rmargin = 3cm, tmargin = 3cm, bmargin = 2cm]
{geometry} \usepackage[russian, english]{babel} 
\usepackage{xcolor}
\usepackage{tikz-cd} 
\usepackage{cleveref}
\usepackage{multicol}
\usepackage{lipsum}
\usepackage{mwe}
\usepackage{subcaption}
\usepackage[colorinlistoftodos]{todonotes}
\crefformat{footnote}{#2\footnotemark[#1]#3}

\usepackage{relsize}

\hypersetup{
    colorlinks,
    linkcolor={blue},
    citecolor={blue},
    urlcolor={blue}
}

\newcommand\blfootnote[1]{
  \begingroup
  \renewcommand\thefootnote{}\footnote{#1}
  \addtocounter{footnote}{-1}
  \endgroup
}

\newtheorem{lemma}{Lemma}
\newtheorem{theorem}{Theorem}
\newtheorem*{theorem*}{Theorem}

\newtheorem{prop}{Proposition}

\newtheorem{defin}{Definition}

\newtheorem{rem}{Remark}

\newtheorem*{hypo*}{Conjecture}
\newtheorem{example}{Example}

\makeatletter\def\blfootnote{\gdef\@thefnmark{}\@footnotetext}\makeatother

\title{An Analogue of Rogers' Theorem on Sieving in Commutative Rings}
\author{Petr Kucheriaviy}
\date{}

\begin{document}

\maketitle

\begin{abstract}
We prove that an analogue of Rogers' theorem on sieving holds for an order if and only if the order is a Dedekind domain. We also prove that it holds for a finite commutative ring if and only if the ring is a direct product of local rings with linearly ordered ideals. 
\end{abstract}

\blfootnote{Support from the Basic Research Program of HSE University is gratefully acknowledged.}

\section{Introduction}

In a recent blog post \cite{tao-rogers-theorem}, Terence Tao drew attention to the following theorem by Rogers.
\begin{theorem*}[Rogers' theorem]
Let $H_1, \ldots, H_r$ be arithmetic progressions that contain $0$. Then 
\[
d \left( \bigcup_{j = 1}^r (a_j + H_j) \right) \ge d \left( \bigcup_{j = 1}^r H_j \right)
\]
for any integers $a_1, \ldots, a_r$. Here $d$ denotes the arithmetic density. 
\end{theorem*}

In other words, if we remove the congruence classes $a_i \pmod{q_i}, i = 1, \ldots, r$, from $\mathbb{Z}$, the density of the resulting set is maximised when $a_1 = \ldots = a_r = 0$.

This theorem appears in \cite[p. 242-244]{halb-sequences} in the book by Halberstam and Roth and was communicated to the authors by Rogers. An equivalent result was proved independently by Simpson in \cite[Lemma 2.3]{Simpson}.

Tao's proof of Rogers' theorem in \cite{tao-rogers-theorem} motivates this note.

\begin{defin}
Let $A$ be an abelian group, and let $H_1, \ldots, H_r$ be subgroups of $A$. Let $H := \bigcap_{j = 1}^r H_j$. Denote by $\pi$ the natural projection from $A$ onto $A/H$.
We say that in $A$ the collection of subgroups $\{ H_1, \ldots, H_r \}$ satisfies the condition $(Rg)$ if the following two conditions hold.

$(Rg_0)$ $A/H$ is finite. 

$(Rg_1)$ For all $a_1, \ldots, a_r \in A$ we have
\[
\left| \pi \left( \bigcup_{j = 1}^r (a_j + H_j) \right) \right| \ge \left| \pi \left( \bigcup_{j = 1}^r H_j \right) \right|.
\]
\end{defin}

The condition $(Rg_1)$ formulated in this way makes sense only if $(Rg_0)$ is satisfied. 
Of course, the condition $(Rg_0)$ is automatically satisfied in finite abelian groups.

\begin{prop}
Let $A$ be an abelian group and suppose that $\{ H_1, H_2 \}$ satisfies $(Rg_0)$. Then the collection of subgroups $\{ H_1, H_2 \}$ satisfies $(Rg)$.
\end{prop}
\begin{proof}
It is enough to prove that 
$|\pi((a_1 + H_1) \cap (a_2 + H_2))| \le |\pi(H_1 \cap H_2)|$.
If $b \in (a_1 + H_1) \cap (a_2 + H_2)$, then $(a_1 + H_1) \cap (a_2 + H_2) = b + H_1 \cap H_2$. Hence $|\pi((a_1 + H_1) \cap (a_2 + H_2))|$ equals $|\pi(H_1 \cap H_2)|$ or zero. 
\end{proof}

\begin{example} \label{key example}
Let $\mathbb{F}_q$ be a finite field with $q$ elements. Let $H_1, H_2, H_3$ be distinct one-dimensional subspaces of $\mathbb{F}_q^2$. Then $\{ H_1, H_2, H_3 \}$ does not satisfy $(Rg)$. Indeed, let $v \in \mathbb{F}_q^2 \setminus H_2$. Then
\[
|H_1 \cup (H_2 + v) \cup H_3| = 3q - 3 < |H_1 \cup H_2 \cup H_3| = 3q - 2.
\]
\end{example}

Let us define a Dedekind domain as an integrally closed one-dimensional Noetherian domain.
Several equivalent definitions of a Dedekind domain are given in \cite[Theorem 9.3]{atiyah-commalg}. 
Let $K \, | \, \mathbb{Q}$ be an algebraic number field of degree $n$. An order of $K$ is a subring $\mathcal{O}$ of $\mathcal{O}_K$ which contains an integral basis of length $n$. 

\begin{prop}
Let $\mathcal{O}$ be an order. Then $(Rg_0)$ is satisfied for any finite set of nonzero ideals in $\mathcal{O}$. 
\end{prop}
\begin{proof}
Since any order $\mathcal{O}$ is an integral domain, the intersection of any finite set of nonzero ideals in $\mathcal{O}$ is a nonzero ideal. 
Let $I$ be a nonzero ideal. Let $a \in I$ be a nonzero element and let $f$ be its characteristic polynomial. The constant term of $f$ lies in $I \cap \mathbb{Z}$, and hence $J = I \cap \mathbb{Z}$ is a nonzero ideal in $\mathbb{Z}$. 
Thus $\mathcal{O}/I$ is a module of finite rank over $\mathbb{Z} / J$ and hence is finite. 
\end{proof}

\begin{theorem} \label{Rogers for orders}
    Let $\mathcal{O}$ be an order. Then $\mathcal{O}$ is a Dedekind domain if and only if any finite set of nonzero ideals in $\mathcal{O}$ satisfies $(Rg)$.
\end{theorem}

\begin{example}
    Let $R = \mathbb{Z}[2i]$. Since $R$ is not a Dedekind domain, Theorem \ref{Rogers for orders} says that there should be a set of ideals of $R$ which does not satisfy $(Rg)$. Indeed, let $I_1 = (2), I_2 = (2i), I_3 = (2 + 2i, 4)$. Then $I_2 \not\subset I_1 \cup I_3$, but $2 + I_2 \subset I_1 \cup I_3$. Hence $I_1, I_2, I_3$ do not satisfy $(Rg)$.
\end{example}

The following reformulation of Rogers' theorem appears in \cite{tao-rogers-theorem}.

\begin{theorem*}
    In a finite cyclic abelian group, $(Rg)$ holds for all finite sets of subgroups.
\end{theorem*}

We prove the analogous result for finite rings. 
\begin{theorem} \label{Rogers for finite rings}
Let $R$ be a finite commutative ring. Then any finite set of ideals in $R$ satisfies $(Rg)$ if and only if $R$ is a direct product of local rings with linearly ordered ideals.
\end{theorem}

\begin{rem}
It follows from the proofs that in Theorem \ref{Rogers for orders} and Theorem \ref{Rogers for finite rings}, it is sufficient to assume that $(Rg)$ holds for sets of three ideals.
\end{rem}

\section{Proof of Theorem \ref{Rogers for finite rings}}

Note that in finite abelian groups $(Rg_1)$ is equivalent to

$(Rg_1')$: For all $a_1, \ldots, a_r \in A$
\[
\left|  \bigcup_{j = 1}^r (a_j + H_j)  \right| \ge \left| \bigcup_{j = 1}^r H_j \right|.
\]

\begin{lemma} \label{Rogers for direct prod}
    Let $R$ be a finite commutative ring and suppose that $R = R_1 \times \ldots \times R_k$. Then $(Rg)$ holds for any set of ideals in $R$ if and only if it holds for any set of ideals in $R_i$ for any $i$.
\end{lemma}
\begin{proof}
We argue by induction on $k$. The case $k = 1$ is trivial, and the inductive step is reduced to the case $k = 2$ via the decomposition $R = (R_1 \times \ldots \times R_{k - 1}) \times R_k$.
Thus it is enough to prove the statement for $k = 2$. We have $R = R_1 \times R_2$.
If $(Rg)$ holds in $R$ for any ideals, then obviously it holds in $R_1$ and $R_2$. Now suppose that $(Rg)$ holds for ideals in each $R_i$. 
Let $I_1, I_2$ be ideals in $R$. Denote by $\varphi_i : R \to R_i$ the canonical projection. 
We have
\[
(a_1 + I_1) \cup (a_2 + I_2) = \left( \varphi_1(a_1 + I_1) \times \varphi_2(a_1 + I_1) \right) \cup \left( \varphi_1(a_2 + I_2) \times \varphi_2(a_2 + I_2) \right).
\]
Let $x \in R_1$. Let $s(x) := \{ j : x \in \varphi_1(a_j + I_j) \}$.
Then
\begin{multline*}
\left| (a_1 + I_1) \cup (a_2 + I_2)  \right| = 
 \left| \bigcup_{x \in R_1} \{x \} \times \bigcup_{j \in s(x)} \varphi_2(a_j + I_j)
  \right| = \sum_{x \in R_1} \left| \bigcup_{j \in s(x)} \varphi_2(a_j + I_j)
  \right|  \ge 
 \\
  \sum_{x \in R_1} \left| \bigcup_{j \in s(x)} \varphi_2(I_j)
  \right| =
  \left| \bigcup_{x \in R_1} \{x \} \times \bigcup_{j \in s(x)} \varphi_2(I_j)
  \right| 
  = |(b_1 + I_1) \cup (b_2 + I_2)|,
\end{multline*}
where $b_1 = (\varphi_1(a_1), 0) \in R_1 \times R_2$, $b_2 = (\varphi_1(a_2), 0) \in R_1 \times R_2$. Repeating the argument, we find that $|(b_1 + I_1) \cup (b_2 + I_2)| \ge |I_1 \cup I_2|$.

Hence any finite set of ideals in $R$ satisfies $(Rg_1')$ and hence satisfies $(Rg)$.
\end{proof}

\begin{proof}[Proof of Theorem \ref{Rogers for finite rings}]
Let $R$ be a finite ring. It can be uniquely (up to isomorphism) decomposed into a finite direct product of local rings \cite[Theorem 8.7]{atiyah-commalg}. 
Thus Lemma \ref{Rogers for direct prod} implies that it is enough to prove Theorem \ref{Rogers for finite rings} for finite local rings.

Let $R$ be a finite local ring. We want to show that any finite set of ideals in $R$ satisfies $(Rg)$ if and only if ideals of $R$ are totally ordered. 
Let us prove this statement by induction on the number of elements of $R$.

Suppose that the ideals of $R$ are totally ordered. After relabelling so that $H_1 \subset \ldots \subset H_r$, the right-hand side of $(Rg_1')$ equals $|H_r|$, while the left-hand side is at least $|a_r + H_r| = |H_r|$. We see that $(Rg_1')$ and hence $(Rg)$ are satisfied.

Suppose that the ideals of $R$ are not totally ordered. First, suppose that $R$ has a unique minimal ideal $J_0$. Then any nonzero ideal of $R$ contains $J_0$. Hence the ideals of $R/J_0$ are not totally ordered and by the induction hypothesis there exist ideals $I_1, \ldots, I_r$ in $R / J_0$ that do not satisfy $(Rg)$. Let $\varphi : R \to R/J_0$ be the canonical projection. It is easy to see that $\varphi^{-1}(I_1), \ldots, \varphi^{-1}(I_r)$ do not satisfy $(Rg)$.

Finally suppose that $R$ has more than one minimal ideal. Let $\mathfrak{m}$ be the unique maximal ideal of the local ring $R$. Let $J = \operatorname{Ann}(\mathfrak{m})$ be the annihilator of $\mathfrak{m}$. The ideal $J$ is a vector space over $R / \mathfrak{m}$ and its one-dimensional subspaces are the minimal ideals. 
Let $(a)$ be a minimal ideal in $R$ (it is clear that minimal ideals are principal). Suppose that $\mathfrak{m} (a) \ne (0)$. Then $\mathfrak{m} (a) = (a)$ since $\mathfrak{m} (a) \subset (a)$ and $(a)$ is minimal. By Nakayama's lemma \cite[Proposition 2.6]{atiyah-commalg} $(a) = 0$, which leads to a contradiction. Hence $(a) \subset J$. This shows that minimal ideals are precisely the one-dimensional vector subspaces of $J$. By our hypothesis there exist at least two minimal ideals and hence $J$ has dimension at least $2$ over $R/\mathfrak{m}$. Let $H$ be a $2$-dimensional subspace in $J$. Let $I_1, I_2, I_3$ be distinct one-dimensional subspaces of $H$. They do not satisfy $(Rg)$ by Example \ref{key example}.
\end{proof}

\section{Proof of Theorem \ref{Rogers for orders}}

\begin{lemma} \label{Dedekind criterion}
Let $R$ be a one-dimensional Noetherian domain. Then $R$ is a Dedekind domain if and only if for any nonzero ideal $I$, the ring $R/I$ is a finite direct product of local rings with linearly ordered ideals.
\end{lemma}

\begin{proof}
    ($\Rightarrow$) Suppose that $R$ is a Dedekind domain and let $I$ be a nonzero ideal. Then $I = \mathfrak{p}_1^{\alpha_1} \ldots \mathfrak{p}_k^{\alpha_k}$ for some prime ideals $\mathfrak{p}_1, \ldots, \mathfrak{p}_k$, and this representation is unique \cite[Corollary 9.4]{atiyah-commalg}. By the Chinese remainder theorem \cite[Proposition 1.10]{atiyah-commalg}, we have
    \[
    R / I \cong R / \mathfrak{p}_1^{\alpha_1} \oplus \ldots \oplus R / \mathfrak{p}_k^{\alpha_k}.
    \]
    Thus it is enough to show that for any $k$ and any prime ideal $\mathfrak{p}$ the ring $R / \mathfrak{p}^k$ is local with linearly ordered ideals. The ideals of $R / \mathfrak{p}^k$ are in a one-to-one order-preserving correspondence with ideals of $R$ that contain $\mathfrak{p}^k$ \cite[Proposition 1.1]{atiyah-commalg}. Hence they are of the form $\mathfrak{p}^{\alpha}$ and hence $R / \mathfrak{p}^k$ is a local ring with linearly ordered ideals.

    ($\Leftarrow$) 
    We need to show that $R$ is integrally closed, assuming that for any nonzero ideal $I$, the ring $R/I$ is a finite direct product of local rings with linearly ordered ideals.
    By \cite[Proposition 5.13]{atiyah-commalg}, $R$ is integrally closed if and only if the localization $R_{\mathfrak{p}}$ is integrally closed for each nonzero prime ideal $\mathfrak{p}$ in $R$. Let us prove that the ideals of $R_{\mathfrak{p}}$ are totally ordered. This is equivalent to $R_{\mathfrak{p}}$ being a valuation ring \cite[Chapter 5, Exercise 28]{atiyah-commalg}, which
    implies that $R_{\mathfrak{p}}$ is integrally closed \cite[Proposition 5.18]{atiyah-commalg}. 

    Let $I$ and $J$ be ideals in $R_{\mathfrak{p}}$. We want to show that $I \subset J$ or $J \subset I$. We can suppose that $I, J$ are nonzero ideals, otherwise the statement is trivial. The ring $R_{\mathfrak{p}}$ is one-dimensional \cite[Proposition 3.11 iv)]{atiyah-commalg} and hence $\mathfrak{p} R_{\mathfrak{p}}$ is its unique nonzero prime ideal. The radical of $I$ equals the intersection of prime ideals that contain $I$ \cite[Proposition 1.14]{atiyah-commalg} and hence equals $\mathfrak{p}R_{\mathfrak{p}}$. In a Noetherian ring, any ideal contains a power of its radical \cite[Proposition 7.14]{atiyah-commalg}. Hence there is an integer $k$ such that $\mathfrak{p}^k R_{\mathfrak{p}} \subset I \cap J$. 
    Let $K$ be a proper ideal of $R$ that contains $\mathfrak{p}^k$. 
    Then $K$ is contained in $\mathfrak{p}$. Indeed, otherwise there exists $a \in K \setminus \mathfrak{p}$. Then $(a) + \mathfrak{p} = (1)$. Hence $(1) = (a) + \mathfrak{p}^{k} \subset K$, which is a contradiction.
    This shows that the ideals of $R_{\mathfrak{p}} / \mathfrak{p}^k R_{\mathfrak{p}}$ are in one-to-one order-preserving correspondence with the ideals of $R / \mathfrak{p}^k$, which are totally ordered by assumption. Hence $I \subset J$ or $J \subset I$.

\end{proof}

It is easy to see that in integral domain $(Rg)$ is satisfied for any finite set of ideals if and only if $(Rg)$ is satisfied for each quotient ring $R/I$ modulo a nonzero ideal $I$. Any order is a one-dimensional Noetherian domain \cite[Chapter I, Proposition 12.2]{neukirch-algnt}.
Therefore Theorem \ref{Rogers for orders} follows from Theorem \ref{Rogers for finite rings} and Lemma \ref{Dedekind criterion}.

\qed

\end{document}